\documentclass[reqno]{amsart}

\usepackage[foot]{amsaddr}

\makeatletter
\@namedef{subjclassname@2020}{\textup{2020} Mathematics Subject Classification}
\makeatother

\usepackage{amssymb} 
\usepackage{amsmath, amsthm} 
\usepackage{graphicx}
\usepackage[top=3cm, bottom=3cm, left=3cm, right=3cm]{geometry}

\usepackage{tikz}
\usetikzlibrary{intersections}
\usetikzlibrary{decorations.markings}
\usetikzlibrary{arrows}
\usetikzlibrary{positioning}

\usepackage{enumitem}
\usepackage{cleveref} 

\usepackage[T1]{fontenc}
\DeclareMathAlphabet{\mathpzc}{OT1}{pzc}{m}{it}
\usepackage{yfonts}
\usepackage[sans]{dsfont}
\usepackage{txfonts}
\usepackage[mathscr]{euscript}
\usepackage{bbm}

\newtheorem{thm}{Theorem}

\theoremstyle{definition}		

\newcommand{\ff}[1]{{\mathbb F}_{#1}}		
\newcommand{\ffs}[1]{{\mathbb F}_{#1}^\star}	

\begin{document}
\title[Observations\dots]{Some observations on bent and planar functions}

\author[R.S. Coulter]{Robert S. Coulter}

\address[R.S. Coulter]{Department of Mathematical Sciences, University of Delaware,
Newark, DE 19716, United States of America.}

\author[S. Senger]{Steven Senger}

\address[S. Senger]{Department of Mathematics, Missouri State University, MO 65897, United States of America.}

\email[R.S. Coulter]{coulter@udel.edu}
\email[S. Senger]{stevensenger@missouristate.edu (corresponding author)}

\thanks{Orcid IDs: (R.S. Coulter) 0000-0002-1546-8779,
(S. Senger) 0000-0003-2912-4464}

\dedicatory{Dedicated to the memory of Pancho (1968--2019)}

\begin{abstract}
We show that the graph of a bent function is a Salem set in an appropriate
sense.
We also establish a simple result that quantifies redundancies in the
difference operators of a function, which applies to bent
functions over fields of odd characteristic via their equivalence to
perfect non-linear functions in that setting.
We end by demonstrating, by entirely elementary means, that the distance between
two distinct planar functions must be at least two.
\end{abstract}

\maketitle


\section{Motivation}

In this note we are interested in three related topics: bent functions, the
inherent redundancy in the difference operators of a function, and the distance
between two functions. We prove a result about each of these.

While much is known about the Fourier properties of the graphs of various types of functions, we prove here that bent functions have graphs whose Fourier behavior is extremal in the sense that they are Salem sets. This result follows directly from the definition of bent functions. 

We next consider difference operators. These functions occur in the study
of perfect non-linear and almost perfect non-linear functions. The former class
of functions includes both planar and bent functions in odd characteristic.
Here, we are interested in the inherent inter-dependence of the difference 
operators. We prove that all difference operators of a function can be
completely determined by a small subset of them, effectively those difference
operators defined by an appropriate basis.

Finally, one major open problem in the study of many classes of functions is related to the minimum distance between two distinct functions within a given class. For example, it is not known if two APN functions can differ in exactly one output. Motivated by this, we use entirely elementary techniques to prove that if two planar functions differ in at least one output, then they must differ in more. A stronger result is known \cite{CK}, but not yet widely available. Our argument merely uses the definition of planarity, and so might give a different perspective for studying the minimum distance of other classes of functions.

\subsection{Definitions and Notation}

Throughout, $q=p^\ell$ for some prime $p$ and $\ell\in\mathbb N$.
We use $\ff{q}$ to denote a finite field of $q$ elements and $\ffs{q}$ its
non-zero elements.
We also let $\vec 0_d$ denote the origin, or all zeros vector, of
$\mathbb F_q^d$.

Let $f\in\ff{q}^d[x]$.
\begin{itemize}
\item For any $a\in\ff{q}^d$ with
$a\ne \vec 0_d$, we define the {\it difference} or {\it difference operator} of
$f$ with respect to $a,$ written $\Delta_{f,a}:\ff{q}^d\rightarrow\ff{q}$, by
$\Delta_{f,a}(x)=f(x+a)-f(x)$.

\item $f$ is called {\it perfect non-linear (PN)}
if and only if for every $a\ne \vec 0_d$, the difference operator
$\Delta_{f,a}$ is equidistributive on all of $\ff{q}.$ That is, every
element of $\ff{q}$ has $q^{d-1}$ pre-images under $\Delta_{f,a}$.
When $d=1$, $f$ is also known as a {\it planar} function.

\item The {\it Fourier transform} of $f$ at $m\in \mathbb F_q^d$ is defined to
be
$$\widehat f(m):=q^{-d}\sum_{x\in \mathbb F_q^d} f(x) \chi(-x \cdot m),$$
where $\chi$ is a nontrivial additive character, and $\cdot$ is the usual inner product. 

\item $f$ is called {\it bent} if and only if for all nontrivial additive
characters $\chi,$ we have that
$$\left | \sum_{x\in\mathbb F_q^d}\chi(f(x)-x\cdot m)\right | = q^\frac{d}{2}.$$
If $p$ is odd, then $f$ is bent if and only if it is PN, see
\cite{A94}, \cite{CM97}. For more general information see \cite{Toka2015}.
\end{itemize}
Following Iosevich and Rudnev in \cite{IR07}, we say that a subset
$E \subseteq \mathbb F_q^d$ is called {\it Salem} if and only if for all $m \in \mathbb F_q^d,$ with $m \neq \vec 0_d,$ we have that there must be some positive constant $C$ independent of $q$ so that
$$\left | \widehat E(m)\right| \leq Cq^{-d}|E|^\frac{1}{2},$$
where $E(x)$ denotes the indicator function of the set $E.$ In this setting, we are typically interested in families of sets that can be defined as $q$ grows large. For more on this definition, and its generalizations, see work by Fraser \cite{Fraser24}.

Given two functions $f,g:A\rightarrow B,$ we say their {\it Hamming distance} or just {\it distance} is the number of images on which they differ. Specifically, we define it as
\[|\{a\in A:f(a)\neq g(a)\}|.\]

\section{Salem sets from graphs}

Salem sets of various kinds have been used to benchmark results in Fourier analysis, geometric measure theory, and other areas, as they exhibit many extreme behaviors. They are an object of study for their own sake, as well as for applications, as can be seen in the aforementioned \cite{Fraser24} and the references contained therein. The important paper by Cheong, Ge, Koh, Pham, Tran, and Zhang \cite{CGKPTZ} is a timely example of how central Salem sets are in geometric combinatorics over finite fields.

By definition, pseudorandom sets are Salem, but deductive constructions are somewhat harder to come by. Moreover, while there are known families of Salem sets, what we prove here is that the graphs of bent functions are not only Salem, but they satisfy the definition with the constant $C =1,$ while many constructions have bigger constants. This follows a theme with other functions widely used in cryptography. The paper by Mihaila and Thornburgh, \cite{MihailaThornburgh}, and the references therein, have results where the graphs of various types of functions have similar Fourier properties.

\begin{thm}
Given a bent polynomial $f\in\mathbb F_q^{d-1}[x],$ its graph $$\left\{(x, f(x)):x\in \mathbb F_q^{d-1}\right\}$$ is a Salem set in $\mathbb F_q^d.$
\end{thm}
\begin{proof}
Suppose $f\in\mathbb F_q^{d-1}[x]$ is bent, and call its graph $E$, so
$$E:= \left\{(x, f(x)):x\in \mathbb F_q^{d-1}\right\}\subseteq\mathbb F_q^d.$$
By definition, $|E|=q^{d-1}$. For convenience, given an $x\in \mathbb F_q^d,$ let $x'$ denote the first $(d-1)$ coordinates,
$$x':= (x_1, x_2, \dots, x_{d-1}).$$
Now, pick an arbitrary $m\in\mathbb F_q^d,$ with $m\neq \vec 0_d,$ and compute the Fourier transform of the indicator function
\begin{align*}
\left| \widehat E(m) \right| &= \left|q^{-d} \sum_{x\in\mathbb F_q^d}E(x)\chi(-x\cdot m)\right|\\
&=q^{-d}\left|\sum_{x'\in \mathbb F_q^{d-1}}\chi(-x_1m_1-x_2m_2-\dots-x_{d-1}m_{d-1}-f(x')m_d)\right|,
\end{align*}
because the terms in the sum where $x_d \neq f(x')$ are not in the set $E$, and therefore contribute zero to the sum. We now split into cases depending on whether $m_d=0$ or not.

{\bf Case 1:} If $m_d=0,$ then $m'\neq \vec 0_{d-1}$ because we know that $m \neq \vec 0_d.$ Continuing as above, we get
\begin{align*}
\left| \widehat E(m) \right| &=q^{-d}\left|\sum_{x'\in \mathbb F_q^{d-1}}\chi(-x_1m_1-x_2m_2-\dots-x_{d-1}m_{d-1})\right|\\
&=q^{-d}\left|\sum_{x'\in \mathbb F_q^{d-1}}\chi(-x'\cdot m')\right|,
\end{align*}
where the sum is zero by orthogonality, because $m' \neq \vec 0_{d-1}.$ Therefore, the Salem bound holds when $m_d=0.$

{\bf Case 2:} If $m_d\neq 0,$ then we can set $m'':= \left(-m_d \right)^{-1}m'$, and compute
\begin{align*}
\left| \widehat E(m) \right| &=q^{-d}\left|\sum_{x'\in \mathbb F_q^{d-1}}\chi\left(\frac{m_d}{m_d}\left(-x_1m_1-x_2m_2-\dots-x_{d-1}m_{d-1}-f(x')m_d\right)\right)\right|\\
&=q^{-d}\left|\sum_{x'\in \mathbb F_q^{d-1}}\chi\left(m_d\left(f(x')+x'\cdot m''\right)\right)\right|.
\end{align*}
Now, this is a sum in our original additive character $\chi,$ but the factor of $m_d$ in the argument of $\chi$ can be thought of as changing this to a sum over a different additive character, $\chi_2$. Because $m_d\neq 0,$ we know that $\chi_2$ will not be a trivial additive character. Finally we set $m'''=(-m''),$ and we see that we can apply the definition of bent functions as follows:
\begin{align*}
\left| \widehat E(m) \right| &=q^{-d}\left|\sum_{x'\in \mathbb F_q^{d-1}}\chi\left(m_d\left(f(x')+x'\cdot m''\right)\right)\right|\\
&=q^{-d}\left|\sum_{x'\in \mathbb F_q^{d-1}}\chi_2\left(f(x')-x'\cdot m'''\right)\right|=q^{-d}q^\frac{d-1}{2}=q^{-d}|E|^\frac{1}{2}.
\end{align*}
\end{proof}


\section{What defines a bent function?}

We next make an observation concerning bent functions in odd characteristic, or
more specifically about the redundancy in difference operators of a function,
which has implications about what defines a bent function.
We may view $\ff{q}^d$ as a $d\ell$-dimensional vector space over $\ff{p}$.
Let $\{g_1,g_2,\ldots,g_{d\ell}\}\subseteq\ff{q}^d$ be an arbitrary but fixed
basis of $\ff{q}^d$ over $\ff{p}$.
There are, of course, many such sets, but one obvious set
is the set $\{\beta_i e_j\}$ where $\{\beta_1,\ldots,\beta_\ell\}$ is basis
for $\ff{q}$ over $\ff{p}$ and the $\{e_j\}$ is the standard set of basis
vectors for $\ff{q}^d$ over $\ff{q}$ with a 1 in the $j$-th coordinate, and
0 in all other positions.
\begin{thm}
Let $a=\sum_{i=1}^{d\ell} k_i g_i$ with $0\le k_i< p$.
For $2\le i\le n$, define $b_i$ by
$b_i=\sum_{j=1}^{i-1} k_j a_j$ and set $b_1=0$.
If $f\in\ff{q}^d[x]$, then 
$$\Delta_{f,a}(x) =
\sum_{i=1}^{d\ell} \sum_{j=1}^{k_i} \Delta_{f,g_i}(x+b_i+jg_i).$$
In particular, every difference operator of $f$ is completely determined
by the $d\ell$ difference operators $\Delta_{f,g_i}$.
\end{thm}
\begin{proof}
We first note that for arbitrary $b,c\in\ff{q}^d$ with $b,c\ne \vec 0_d$, we
have
$\Delta_{f,b+c}(x)=f(x+b+c)-f(x)=f(x+b+c)-f(x+b)+f(x+b)-f(x)$.
Thus,
\begin{equation} \label{combine}
\Delta_{f,b+c}(x) = \Delta_{f,c}(x+b) + \Delta_{f,b}(x).
\end{equation}
In particular, for arbitrary non-zero $b\in\ff{q}^d$ and integer
$0\le k < p$, we have
\begin{equation} \label{kbeq}
\Delta_{f,kb} = \sum_{i=1}^k \Delta_{f,b}(x+ib).
\end{equation}

Next, set $c=\sum_{i=1}^n c_i$ with $c_i\in\ff{q}^d$, and for 
$2\le i\le n$, define $b_i$ by
$b_i=\sum_{j=1}^i c_j$ and set $b_1=0$.
Then repeated application of \Cref{combine} also yields
\begin{equation} \label{allbut}
\Delta_{f,c}(x) =
\sum_{i=1}^n \Delta_{f,c_i}(x+b_i).
\end{equation}
Finally, we now turn to $a=\sum_{i=1}^{d\ell} k_i g_i$.
Set $a_i=k_i g_i$ for $1\le i\le d\ell$.
Starting with \Cref{allbut}, we appeal to \Cref{kbeq} to have
\begin{align*}
\Delta_{f,a}(x) &= \sum_{i=1}^n \Delta_{f,a_i}(x+b_i)\\
&= \sum_{i=1}^n \Delta_{f,k_i g_i}(x+b_i)\\
&= \sum_{i=1}^n \sum_{j=1}^{k_i}\Delta_{f,g_i}(x+b_i+ j g_i),
\end{align*}
as claimed.
The final claim is immediate, since
any element of $\ff{q}^d$ can be written in the form of the $a$ in the
statement.
\end{proof}
As we noted in the introduction, this result can be applied specifically
to PN functions, as well as APN and planar functions. Since PN functions and
bent functions are equivalent in odd characteristic, it also applies for
bent functions in that scenario.
The result can be seen as a measure of redundancy in the difference operators
of a function; the $q^{d\ell}-1$ non-trivial difference operators of a 
function $f\in\ff{q}^d[x]$ are completely determined by only $d\ell$ of them.

\section{Minimum distance between distinct planar functions}

Here we discuss the minimum distance between two distinct planar functions. While the first listed author and Kaleyski have results on the minimum distance between members of various function classes, including planar functions, in \cite{CK}, this argument is completely elementary. As such, it could guide other such minimum distance estimates. For example, one big conjecture (see \cite{MihailaThornburgh} and references contained therein) is that distinct APN functions should have a minimum distance strictly greater than 1. We now present an elementary argument showing that the analogous statement does indeed hold for planar functions.

\begin{thm} Suppose $q=p^\ell$ with $p>3$ prime, and $\ell \in \mathbb N.$ Given two distinct planar functions $f,g:\mathbb F_q\rightarrow F_q$, there must be at least two distinct $x_1,x_2\in\mathbb F_q$ so that $f(x_1)\neq g(x_2)$ and $f(x_2)\neq g(x_2).$
\end{thm}
\begin{proof}
By way of contradiction, suppose there is some $w\in\mathbb F_q$ so that $f(w)\neq g(w),$ but that for all $x\in\mathbb F_q\setminus\{w\},$ we have $f(x)=g(x).$ So for $a\neq 0,$ and $x\neq w-a,$ we will have
\begin{equation}\label{shiftEq}
f(x+a)=g(x+a).
\end{equation}
By the definitions of differences, we will have that
\[\Delta_{f,a}(\mathbb F_q\setminus \{w,w-a\}) = \Delta_{g,a}(\mathbb F_q\setminus \{w,w-a\}).\]
Since both $f$ and $g$ are planar, their difference operators are bijections, so this gives us that
\begin{equation}\label{spittinImages}
\Delta_{f,a}(\{w,w-a\}) = \Delta_{g,a}(\{w,w-a\}).
\end{equation}
Recalling \eqref{shiftEq} and the fact that $a\neq 0,$ we see that $f(w+a) = g(w+a).$ This shows us that $f(w+a)-f(w)\neq q(w+a)-g(w).$ So by \eqref{spittinImages}, we must have that
\[\Delta_{f,a}(w) = \Delta_{g,a}(w-a)\Rightarrow\]
\begin{equation}\label{wawa}
f(w+a)-f(w) = g((w-a)+a)-g(w-a)=g(w)-g(w-a).
\end{equation}
Without loss of generality, let $w=0$ and $f(w)=0.$ This is a safe assumption because if $w\neq 0,$ we could just continue considering horizontal shifts of both $f$ and $g.$ Similarly, if $f(w)\neq 0,$ we could consider vertical shifts of both $f$ and $g.$ Now, since $p\neq 2,$ we can let $g(0)=2c$ for some $c\neq 0.$ Since $f(x)=g(x)$ for all nonzero $x,$ the above combine with \eqref{wawa} to give
 \[f(a)-0=g(0)-g(-a)=2c-f(-a)\Rightarrow\]
 \begin{equation}\label{a-a}
 2c = f(a)+f(-a).
 \end{equation}
For all $b\neq 0,$ we know $\Delta_{f,b}$ is a bijection, so there must be a unique $x_b\in\mathbb F_q$ so that $f(x_b+b)-f(x_b) = 0.$ Now apply \eqref{a-a} with $a=x_b$ to get
\[2c = f(x_b)+f(-x_b).\]
Applying \eqref{a-a} again with $a=(x_b+b)$ gives
\[2c = f(x_b+b)+f(-x_b-b).\]
Using these last two equations, we write
\[f(x_b+b)+f(-x_b-b)=f(x_b)+f(-x_b)\Rightarrow\]
\[f(x_b+b)-f(x_b)=f(-x_b)-f(-x_b-b).\]
Since $f(x_b+b)-f(x_b)=0$ by definition of $x_b,$ we are left with
\[0=f(-x_b)-f(-x_b-b).\]
But this can be written as
\[0=f((-x_b-b)+b) - f(-x_b-b)=\Delta_{f,b}(-x_b-b).\]
Since $x_b$ is the unique zero of $\Delta_{f,b}$ by definition, we must have that
\[x_b = -x_b-b \Rightarrow 2x_b = -b \Rightarrow x_b = \frac{-b}{2}.\]
Plugging this back into the definition of $\Delta_{f,b}(x_b)$ gives us
\[0=\Delta_{f,b}(x_b) = f(x_b+b)-f(x_b) = f\left(\frac{-b}{2}+b\right)-f\left(\frac{-b}{2}\right)=f\left(\frac{b}{2}\right)-f\left(\frac{-b}{2}\right).\]
Since this holds for arbitrary $b\neq 0,$ by recalling that $p\neq 2,$ we have that in general, for all $a=\frac{b}{2}\neq 0,$
\begin{equation}\label{symmetry}
f(a)=f(-a).
\end{equation}
To finish, recall that \eqref{a-a} guaranteed $f(a)+f(-a)=2c$ for all $a\neq 0.$ But by \eqref{symmetry}, we know that $f(a)=f(-a)$ for all $a\neq 0.$ Combining these would imply that $f(a)+f(a)=2c\Rightarrow f(a)=c$ for all nonzero $a\in\mathbb F_q,$ which cannot be the case, as this would ensure that $\Delta_{f,a}$ has $p-2$ zeros for each nonzero $a.$ Since $p>3,$ this would mean $\Delta_{f,a}$ would not have a unique zero, contradicting the definition of the planarity of $f.$
\end{proof}
We note that there is no doubt that this argument could be pushed further.
For one thing, the conclusion of the argument that yields the contradiction is
that we have a planar function with image set of cardinality 2, which is 
significantly below the well-known lower bound of $(q+1)/2$. This suggests 
that there is much scope here.
However, applying the same argument for distance 2 appears to split into
many subcases, and we feel that the elementary argument given here illustrates
the idea perfectly without any unnecessary complications. With the results
of \cite{CK} also in mind, we have therefore refrained from continuing this
line of inquiry here.


\end{document}